\documentclass[12pt, a4paper]{amsart}
\usepackage[margin=1in]{geometry}
\usepackage{amsmath,amssymb,amsfonts, float}
\usepackage{hyperref}
\usepackage{cleveref}
\usepackage{graphicx}

\usepackage{xcolor}
\usepackage{multicol}
\usepackage{tabularx}
\usepackage{mathpazo}
\newtheorem{thm}{Theorem}[section] 
\newtheorem{prop}[thm]{Proposition}
\newtheorem{lem}[thm]{Lemma}
\newtheorem{cor}[thm]{Corollary}
\theoremstyle{definition}
\newtheorem{definition}[thm]{Definition}
\newtheorem{expl}[thm]{Example}

\newtheorem{rem}[thm]{Remark}

\numberwithin{equation}{section}

\newcommand{\F}{\mathbb{F}}

\DeclareMathOperator{\rank}{rank}
\DeclareMathOperator{\diag}{diag}
\DeclareMathOperator{\diam}{diam}
\DeclareMathOperator{\spn}{span}

\begin{document}
\title{On gcd-graphs over matrix rings}

\author{Dung Nguyen}
\address{Elmhurst University, Elmhurst, IL 60126, USA.}
\email{dnguy9448@365.elmhurst.edu}

\author{Tung T. Nguyen}
\address{Elmhurst University, Elmhurst, IL 60126, USA.}
\email{tung.nguyen@elmhurst.edu}

\begin{abstract}
Graphs defined over finite rings are well studied in the literature. The study of these graphs benefits from rich connections between several areas of mathematics, including number theory, algebra, combinatorics, and graph theory, and these connections often lead to interesting interactions between algebraic and combinatorial structures. In this article, we investigate gcd-graphs defined over matrix rings with coefficients in finite fields.  We show that these graphs exhibit several interesting graph-theoretic properties. Along the way, we also prove some results on the structure of matrix rings, which may be of independent interest.
\end{abstract}

\maketitle

\section{Introduction}
Graphs defined over finite rings are well studied in the literature, tracing their origins to the foundational work of Klotz and Sander on gcd-graphs (see, for example, \cite{unitary,anderson2021graphs, klotz2007some, nguyengcd2026}). Unlike graphs defined over abstract abelian groups, graphs defined over rings benefit from both the additive and multiplicative structures of the underlying ring. Moreover, the interplay between these two structures often gives rise to rich and interesting properties of the associated graphs. To set the stage for our subsequent analysis, we recall the definition of these graphs.
\begin{definition} \label{def:cayley}
Let $R$ be a finite ring and $S \subset R \setminus \{0\}$ be a symmetric subset (meaning that $S=-S$). The Cayley graph $\Gamma(R,S)$ is defined as follows:
\begin{enumerate}
    \item The vertex set is the set of elements in $R$.
    \item Two vertices $x, y \in R$ are adjacent if $x - y \in S$.
\end{enumerate}
We refer to $S$ as the generating set for $\Gamma(R,S)$. In practice, $S$ often has arithmetic origins. 
\end{definition}

One of the earliest and most extensively studied types of Cayley graphs over rings is the class of gcd-graphs, first introduced in \cite{klotz2007some}. We begin by recalling this fundamental concept. Let $n$ be a positive integer and $D$ a subset of proper divisors of $n$. The gcd-graph $G_n(D)$ is defined as follows: (1) the vertices of $G_n(D)$ are elements of the finite ring $\mathbb{Z}/n\mathbb{Z}$, and (2) two vertices $a, b$ are adjacent if $\gcd(a-b, n) \in D$. In this setting, the generating set is precisely the collection of elements of the form $du$, where $d \in D$ and $u \in (\mathbb{Z}/n\mathbb{Z})^{\times}$.

The concept of gcd-graphs can be naturally generalized to any finite ring in which we employ an equivalent characterization: a graph $\Gamma(R,S)$ is a gcd-graph if and only if $S$ is stable under the left and right actions of $R^{\times}$. That is, $S$ satisfies the condition $R^{\times} S R^{\times} = S$, where $R^{\times}$ denotes the group of units in $R$ (see \cite{pst-gcd-new-2,minavc2024gcd, nguyengcd2026, nguyen2026u, pst-gcd-new}).

In this paper, we focus on the case where the underlying ring $R$ is a matrix ring defined over a finite field. In this setting, we can utilize tools from linear algebra to address fundamental questions concerning the structure of the associated gcd-graphs. We summarize here some of our main results; for precise statements and detailed proofs, we refer the reader to the main body of the text.

Let $R=M_n(\F_q)$ where $\F_q$ is a finite field with $q$ elements. Let $S$ be a symmetric subset of $M_n(\F_q)$. Throughout the rest of the paper, we will assume that $n>1$ since in the case $n=1$, a gcd-graph is either a complete graph or an empty graph (additionally, certain rank subadditivity properties and non-commutative matrix decompositions we seek to investigate simply do not exist in a 1-dimensional space). For each $1 \leq k \leq n$, let $V_k$ be the subset of $R$ consisting of all matrices of rank $k.$
\begin{enumerate}
    \item $\Gamma(R,S)$ is a gcd-graph if and only if there exists a subset $I \subset \{1, 2, \ldots, n\}$ such that $S= \bigsqcup_{k \in I} V_k.$ We will write $\Gamma(M_n(\F_q), [k_1, k_2, \ldots, k_r])$ for $\Gamma(M_n(\F_q),S)$ where $k_1<k_2< \cdots <k_r$ are elements of $I.$ To avoid the trivial cases of empty graphs or complete graphs, we will assume that $1 \leq r \leq n-1$.
    \item The gcd-graph $\Gamma(M_n(\F_q), [k_1, k_2, \ldots, k_r])$ is connected with diameter equal to
        \[ \max \left\{2, \left\lceil \frac{n}{k_{r}} \right\rceil \right\}. \]
    \item The gcd-graph $\Gamma(M_n(\F_q), [k_1, k_2, \ldots, k_r])$ is not bipartite. 
    \item The gcd-graph $\Gamma(M_n(\F_q), S)$ is Hamiltonian. 
    \item The clique number of $\Gamma(M_n(\F_q), [1])$ is $q^n$.
    \item   The sum of all matrices in $M_n(\F_q)$ with fixed rank $k$ is the zero matrix.
\end{enumerate}

\begin{figure}[H]
\centering
\includegraphics[width=0.5 \linewidth]{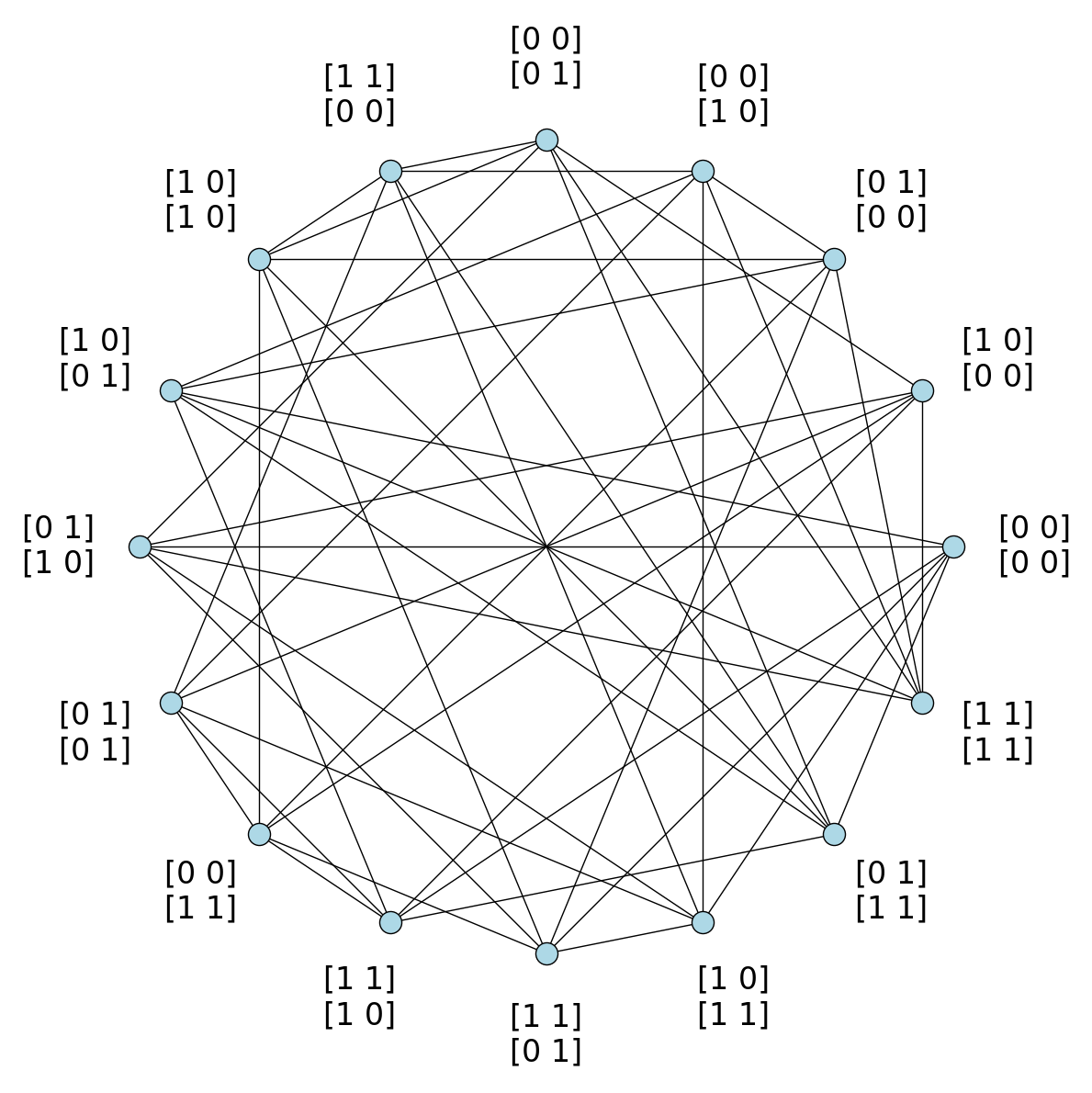}
\caption{An example of gcd-graph $\Gamma(M_2(\F_2),[2])$}
\label{fig:example1}
\end{figure}

\subsection*{Outline} 
In \Cref{sec:background}, we provide some background on graph theory that we will need later in the text. In \Cref{subsec:classification}, we give necessary and sufficient conditions for a Cayley graph over \( M_n(\F_q) \) to be a gcd-graph. These conditions are explicit and can be easily verified. In \Cref{subsec:connectivity}, we describe the connectivity of gcd-graphs. In particular, we compute the distance between any two vertices. In the course of proving our results, we obtain some statements on representing a matrix as a sum of two matrices with fixed ranks---a topic that may be of independent interest. Other fundamental properties, such as bipartiteness, Hamiltonicity, and clique numbers, are discussed in \Cref{subsec:bipartite} and \Cref{subsec:rank1}. Finally, in \Cref{subsec:sum}, we show that the sum of all matrices with fixed rank is \(0\), which leads to an interesting corollary on the walk structure of a gcd-graph.
\subsection*{Future work}
The study of gcd-graphs over matrix rings presents several additional questions that appear to be considerably more subtle. For example, while we determine the clique number of the graph $\Gamma(M_n(\mathbb{F}_q), [1])$, the corresponding problem for gcd-graphs of the form $\Gamma(M_n(\mathbb{F}_q), [k_1, k_2, \ldots, k_r])$ remains open. Other natural directions include the study of chromatic numbers and explicit spectral descriptions (for gcd-graphs over finite commutative rings, these are rather well-studied, see \cite{unitary, chen2022unitary, klotz2007some, nguyengcd2026}). We hope that the results in this paper provide a useful framework for investigating these questions.
\subsection*{Code} 
The code that we wrote to test our calculations and conduct numerical experiments on gcd-graphs can be found at \cite{githubrepo}. 

\section{Background on graph theory} \label{sec:background}
In this section, we recall some background in graph theory that we will need later in this article. 
\subsection{Connectedness of Cayley graphs}
A Cayley graph \(\Gamma(R,S)\) is connected if and only if every element of \(R\) can be written as a sum of elements of \(S\). We may think of the generating set \(S\) as providing the ``bridges'' that connect vertices in \(R\). To illustrate this, suppose we want to find a walk between vertices \(a\) and \(b\) in \(R\). We can construct such a walk by moving through intermediate vertices:
\[
a \to x_1 \to x_2 \to \cdots \to x_{n-1} \to b.
\]
By the definition of a Cayley graph, each edge corresponds to an element of \(S\), so we have
\begin{align*}
x_1 - a &= s_1 \in S, \\
x_2 - x_1 &= s_2 \in S, \\
&\vdots \\
b - x_{n-1} &= s_n \in S.
\end{align*}
Summing these equations telescopically, we obtain
\[
b - a = s_1 + s_2 + \cdots + s_n,
\]
where each \(s_i \in S\). This argument also shows that the distance between \(a\) and \(b\) in the Cayley graph is the smallest integer \(n\) for which \(b-a\) can be expressed in this form.

\subsection{Diameter and clique number of a graph}
The diameter of a graph is the maximum length of the shortest path between any two vertices, namely 
\[ \text{diam}(G)=\max_{u,v \in V(G)} d(u,v). \]
The clique number of a graph, denoted $\omega(G)$, is the number of vertices in its largest complete subgraph, where every vertex is connected to every other vertex in that subset.

\subsection{Bipartite graphs} 
A graph $G = (V, E)$ is called a bipartite graph if the vertex set $V$ can be partitioned into two disjoint subsets $U$ and $W$ such that every edge  $e \in E$ connects a vertex in $U$ to a vertex in $W$. Furthermore, for $u_1, u_2 \in U$ and $w_1, w_2 \in W$, $(u_1, u_2) \not \in E$ and $(w_1, w_2) \not \in E$. In other words, an edge in $V$ can't connect two vertices in the same partition. A graph is bipartite if and only if it does not contain an odd cycle.

\subsection{Hamiltonian graphs}
A Hamiltonian graph is a graph that contains a Hamiltonian cycle, which is a closed cycle that visits every single vertex in the graph exactly once before returning to the starting vertex.

\section{Gcd-graphs over matrix rings} \label{sec:gcd-graphs}
Let $F$ be a finite field of order $q$. We first provide the necessary and sufficient conditions for the Cayley graph $\Gamma(M_n(F),S)$ to be a gcd-graph over $M_n(F).$ Recall that for each $0 \leq k \leq n$, $V_k$ is the subset of all matrices in $M_n(F)$ of rank $k.$ For each positive integer $d$ and $0 \leq k \leq d$, let $I_d(k)$ denote the $d \times d$ diagonal matrix with its first $k$ diagonal entries equal to $1$ and all remaining entries equal to $0$.

\subsection{Classification of gcd-graphs over a matrix ring} \label{subsec:classification}
In this section, we provide the necessary and sufficient conditions for a Cayley graph over $M_n(F)$ to be a gcd-graph. 

\begin{thm}
    Let $S$ be a subset of $M_n(F) \setminus \{0\}.$ Then $S$ is stable under the left and right action of $M_n(F)^{\times} = GL_n(F)$ if and only if $S$ is a disjoint union of some $V_k$'s; namely there exists $1 \leq k_1 < k_2 < \cdots <k_r$ such that 
    \[ S= \bigsqcup_{i=1}^r V_{k_i}.\]
\end{thm}

\begin{proof}
    For sufficiency, let $S$ be a disjoint union of some sets $V_{k_i}$ and $A \in S$. By definition, there exists $k \in \{k_1,k_2,\ldots,k_r\}$ such that $ \rank(A) = k.$ Let $P,Q \in GL_n(F)$ be two arbitrary invertible matrices. We note that multiplication by an invertible matrix corresponds to performing elementary row or column operations, which do not alter the rank of a matrix. Hence,
    \[\rank(PAQ) = \rank(A) = k.\]
    By definition, $PAQ \in V_k$. Since $V_k \subseteq S$, $PAQ \in S$. Thus, $S$ is stable under the left and right action of $GL_n(F)$.\par
    For necessity, let $S \subseteq M_n(F)$ be stable under the left and right action of $GL_n(F)$. Suppose $A \in S$ and $B$ is an arbitrary matrix such that $\rank(A)=\rank(B)=k$. Since left and right multiplication by an invertible matrix do not alter the rank of a matrix, there exist $P_1, Q_1, P_2, Q_2 \in GL_n(F)$ such that $P_1 A Q_1=I_n(k)$ and $B = P_2 I_n(k) Q_2$. It follows that $B = P_2 P_1 A Q_1 Q_2.$ Let $P=P_2P_1$ and $Q=Q_1Q_2$. Since $P_1, Q_1, P_2, Q_2 \in GL_n(F)$, $P,Q \in GL_n(F)$. Thus, if there exists at least one rank $k$ matrix in $S$, the entire set $V_k$ must also be a subset of $S$. We conclude that there exists $1 \leq k_1 <k_2 < \cdots < k_r$ such that $S= \bigsqcup_{i=1}^r V_{k_i}.$
\end{proof}
We have an immediate corollary. 
\begin{cor} \label{cor:generating_set}
    The Cayley graph $\Gamma(M_n(F),S)$ is a gcd-graph if and only if there exists $1 \leq k_1 < k_2 < \cdots <k_r$ such that 
    \[S= \bigsqcup_{i=1}^r V_{k_i}.\]
\end{cor}

\begin{cor}
    $\Gamma(M_n(F), [k_1, \ldots, k_r])$ is a regular graph of degree 
    \[|S| = \sum_{i=1}^r |V_{k_i}|=\sum_{i=1}^r \left( \prod_{j=0}^{k_i-1} \frac{(q^n-q^j)^2}{q^{k_i}-q^j} \right).\]
    Here $q$ is the order of the finite field $F.$
\end{cor}
\begin{proof}
    Recall from Definition~\ref{def:cayley} that in $\Gamma(M_n(F), S)$, a matrix $A$ is adjacent to all matrices $B$ such that $B-A \in S$. For each matrix $X \in S$, there exists a corresponding matrix $B$ such that $B-A=X$.  Hence, $\Gamma(M_n(F), S)$ is a regular graph of degree $|S|$.
    We have shown in Corollary~\ref{cor:generating_set} that $S$ must be a disjoint union of some $V_k$'s. By \cite[Formula 3]{Migler01102006}, we know that
    \[|V_k| = \prod_{j=0}^{k-1} \frac{(q^n-q^j)^2}{q^k-q^j}.\]
    Therefore, $\Gamma(M_n(F), [k_1, \ldots, k_r])$ is a regular graph of degree 
    \[|S| = \sum_{i=1}^r |V_{k_i}|=\sum_{i=1}^r \left( \prod_{j=0}^{k_i-1} \frac{(q^n-q^j)^2}{q^{k_i}-q^j} \right).\]
\end{proof}

\subsection{Automorphisms of a gcd-graph over a matrix ring}
In this section, we describe some automorphisms of a gcd-graph defined over the matrix ring $M_n(F).$
\begin{prop} \label{prop:automorphism}
    Let $G:=\Gamma(M_n(F), S)$ be a gcd-graph defined over $M_n(F).$ 
    \begin{enumerate}
        \item Let $A \in M_n(F).$ Then the translation map $t_{A}: G \to G$ defined by $t_A(X)=A+X$ is a graph automorphism of $G.$
        \item Let $P \in GL_n(F).$ Then the left dilation map $\tau_{P}: G \to G$ defined by $\tau_P(X)=PX$ is a graph automorphism of $G.$ Similarly, the right dilation map $\gamma_P: G \to G$ defined by $\gamma_{P}(A)=AP$ is also a graph automorphism of $G.$
    \end{enumerate}
\end{prop}

\begin{proof}
    The first property is straightforward. The second property follows from the fact that the generating set $S$ is stable under the left and right action of $GL_n(F).$
\end{proof}

\subsection{Connectivity of gcd-graphs over a matrix ring} \label{subsec:connectivity}

In this section, we study the connectivity of a gcd-graph defined over $M_n(F).$ In order to do this, we first recall the following definition. 
\begin{definition} \label{def:set_addition}
    For two subsets $X, Y$ of $M_n(F)$ we define 
    \[ X+ Y = \{x+y \mid x \in X, y \in Y \}.\]
\end{definition}
Before we explain the results and proofs in detail, let us provide an overview of our approach, which expands on the work in \cite[Lemma 5.1]{guterman2021} and is based on the theory of Gaussian elimination together with certain combinatorial arguments. More precisely, in Lemma~\ref{lem:r<m}, Lemma~\ref{lem:r>m}, and Theorem~\ref{thm:sum_rank_sets}, we determine when an \(n \times n\) matrix \(A\) of rank \(r\) can be written as a sum of two \(n \times n\) matrices of ranks \(m\) and \(k\).

We first remark that any such matrix \(A\) can be expressed as \(A = P I_n(r) Q\), where \(P\) is a product of elementary matrices corresponding to row operations, \(Q\) is a product of elementary matrices corresponding to column operations, and \(r\) is the rank of \(A\). Since multiplication by \(P\) and \(Q\) preserves rank, it suffices to study when \(I_n(r)\) can be written as a sum of two matrices from \(V_m\) and \(V_k\). This reduction follows from the observation that
\[
A = P I_n(r) Q = P(X_m + X_k)Q = P X_m Q + P X_k Q,
\]
where \(X_m, P X_m Q \in V_m\) and \(X_k, P X_k Q \in V_k\).

Let \(\diag(d_1, d_2, \ldots, d_n)\) denote the \(n \times n\) diagonal matrix with entries \(d_1, d_2, \ldots, d_n\) on the diagonal and zeros elsewhere. Let \(E_{x,y}\) denote the \(n \times n\) matrix with a \(1\) in the \((x,y)\)-entry and zeros elsewhere. Without loss of generality, we will assume that $m \geq k.$
\begin{lem} \label{lem:r<m}
    If $m-k \leq r \leq m$, any $n \times n$ matrix with rank $r$ is a sum of an $n \times n$ matrix with rank $m$ and an $n \times n$ matrix with rank $k$.
\end{lem}
\begin{proof}
    First, let us suppose that $F$ has more than 2 elements. Then, we may choose an $h \in F \setminus\{0,1\}$ and write $I_n(r)$ as a sum of 2 matrices
    \begin{align*}
        I_n(r)
        &= \diag(\underbrace{1,\ldots,1}_{m-k},\underbrace{1-h,\ldots,1-h}_{r-(m-k)},\underbrace{1,\ldots,1}_{m-r},\underbrace{0,\ldots,0}_{n-m})\\
        &+ \diag(\underbrace{0,\ldots,0}_{m-k},\underbrace{h,\ldots,h}_{r-(m-k)},\underbrace{-1,\ldots,-1}_{m-r},\underbrace{0,\ldots,0}_{n-m}).
    \end{align*}
    Let $X_m, X_k$ be the two diagonal matrices in the equation above. Then, $X_m$ has $m$ non-zero rows and thus has rank $m$ while $X_k$ has $k$ non-zero rows and thus has rank $k$.
    If $F$ is the field with two elements, we further divide the condition into two cases, $m>k$ or $m=k$.
    \paragraph{Case 1: $m>k$}
    Let $U_r$ be an $r\times r$ upper triangular matrix with ones on the main diagonal. We set exactly $r - (m-k)$ ones on the superdiagonal of $U_r$, and the rest to zeros. Because $m > k$, we are guaranteed that $r - (m-k) \le r-1$, ensuring there are enough positions on the superdiagonal. We can see that $U_r$ has rank $r$ and it has the following explicit description 
    \[U_r = I_r +\sum_{i=1}^{r - (m-k)} E_{i,(i+1)}.\]
    We then construct the submatrices $X_m,X_k$ as follows:
    \[
    X_m =
    \begin{bmatrix}
        U_r & 0\\
        0 & I_{m-r}
    \end{bmatrix},
    \]
    \[
    X_k = I_m(r) - X_m =
    \begin{bmatrix}
        I_r-U_r & 0\\
        0 & -I_{m-r}
    \end{bmatrix}.
    \]
    Utilizing the properties of block diagonal matrices,
    \[\rank(X_m) = \rank(U_r) + \rank(I_{m-r}) = r+m-r=m.\]
    The block $I_r - U_r$ is  upper triangular with exactly $r - (m-k)$ non-zero entries on its superdiagonal, giving it a rank of exactly $r - (m-k)$. 
    \[\rank(X_k) = \rank(I_r - U_r) + \rank(-I_{m-r}) = r - (m-k)+m-r=k.\]
    Adding $X_m,X_k$ we get $I_r$ as the only non-zero block of the sum,
    \[
    X_m + X_k =
    \begin{bmatrix}
        U_r + I_r - U_r & 0\\
        0 & I_{m-r} - I_{m-r}
    \end{bmatrix}
    =
    \begin{bmatrix}
        I_r& 0\\
        0 & 0
    \end{bmatrix}.
    \]
    Both $X_m$ and $X_k$ are then padded with zeros to reach the dimension $n\times n$.

    \paragraph{Case 2: $m=k$}
    We consider 3 more subcases, either $r<m$, $r=m \geq 2$, or $r=m=1$.
    \subparagraph{Case 2.1: $r<m$}
    Let $X_m$ be the following $m\times m$ cyclic shift matrix:
    \[X_m = \left[ \sum_{i=1}^{m-1} E_{i,(i+1)} \right] + E_{m,1}.\]
    Evaluating the determinant of $X_m$ by expanding along the first column isolates the single non-zero entry coming from $E_{m,1}$. The corresponding minor is $I_{m-1}$, yielding a non-zero determinant. Thus $X_m$ is invertible and has full rank $m$. Then, let $X_k = I_m(r) - X_m$. Since $r<m$, the ones from $I_m(r)$ do not appear on the $m$-th row of $X_k$. Evaluating the determinant of $X_k$ by expanding along the $m$-th (bottom) row, the only non-zero entry comes from $E_{m,1}$. The corresponding minor is a triangular matrix with ones all across its main diagonal, yielding a non-zero determinant. Thus $X_k$ is also invertible and has full rank $m$. $X_m$ and $X_k$ add up to $I_m(r)$.
    \subparagraph{Case 2.2: $r=m \geq 2$}

If $m=2$, we take the following matrices:
\[
X_m=\begin{pmatrix}
0&1\\
1&1
\end{pmatrix},
\qquad
X_k=\begin{pmatrix}
1&1\\
1&0
\end{pmatrix}.
\]
Both matrices have rank $2$ and add up to $I_2$ over $\mathbb{F}_2$.
Both $X_m$ and $X_k$ are then padded with zeros to reach the
dimension $n\times n$.

We now assume that $m\ge3$. In this case,  we can make one small adjustment to $X_m$:
    \[X_m = \left[ \sum_{i=1}^{m-1} E_{i,(i+1)} \right] + E_{m,1} + E_{m,2}.\]
    Similarly, through the calculation of the determinant with the first column, we can show that $X_m$ has full rank $m$. Then, let $X_k = I_m - X_m$. The matrix $X_k$ has ones on its main diagonal, superdiagonal, and at positions $(m,1)$ and $(m,2)$. Because the first row of $X_k$ contains ones exactly at $(1,1)$ and $(1,2)$ with zeros elsewhere, we can apply the elementary row operation of adding the first row to the $m$-th row. Over $\F_2$, this turns the entries at $(m,1)$ and $(m,2)$ to zero, leaving the $m$-th row with only a single one located on the main diagonal at $(m,m)$. This transforms $X_k$ into an upper triangular matrix with ones completely across its main diagonal, thus has full rank $m$. Because elementary row operations preserve the rank, $X_k$ has full rank $m$. $X_m$ and $X_k$ add up to $I_m$, and thus $I_m(r)$ since $r=m$.
    \subparagraph{Case 2.3: $r=m=1$}
    We default to the following rank 1 matrices $X_m,X_k$:
    \[
    X_m = 
    \begin{bmatrix}
        1 & 1\\
        0 & 0
    \end{bmatrix},
    X_k =
    \begin{bmatrix}
        0 & 1\\
        0 & 0
    \end{bmatrix},
    \]
    \[
    X_m+X_k =
    \begin{bmatrix}
        1 & 0\\
        0 & 0
    \end{bmatrix}.
    \]
    Both $X_m$ and $X_k$ are then padded with zeros to reach the dimension $n\times n$.
\end{proof}
Let us demonstrate Lemma~\ref{lem:r<m} with a few examples. More examples can be found in \cite{githubrepo}.
\begin{expl}
    In the ring $M_5(\F_2)$, we can write $I_5(3)$ as a sum of matrices with rank $m=4$ and rank $k=2$. In this case, $m>k$.
    \[
    X_m =
    \begin{bmatrix}
        \mathbf{1} & 1 & 0 & 0 & 0\\
        0 & \mathbf{1} & 0 & 0 & 0\\
        0 & 0 & \mathbf{1} & 0 & 0\\
        0 & 0 & 0 & 1 & 0\\
        0 & 0 & 0 & 0 & 0
    \end{bmatrix},
    X_k =
    \begin{bmatrix}
        0 & 1 & 0 & 0 & 0\\
        0 & 0 & 0 & 0 & 0\\
        0 & 0 & 0 & 0 & 0\\
        0 & 0 & 0 & 1 & 0\\
        0 & 0 & 0 & 0 & 0
    \end{bmatrix}.
    \]
\end{expl}
\begin{expl}
    In the ring $M_3(\F_2)$, we can write $I_3(2)$ as a sum of matrices with rank $m=k=3$.
    \[
    X_m =
    \begin{bmatrix}
        0 & 1 & 0\\
        0 & 0 & 1\\
        1 & 0 & 0
    \end{bmatrix},
    X_k =
    \begin{bmatrix}
        \mathbf{1} & 1 & 0\\
        0 & \mathbf{1} & 1\\
        1 & 0 & 0
    \end{bmatrix}.
    \]
\end{expl}

\begin{lem} \label{lem:r>m}
    If $m < r \leq \min\{m+k, n\}$, any $n \times n$ matrix with rank $r$ is a sum of an $n \times n$ matrix with rank $m$ and an $n \times n$ matrix with rank $k$.
\end{lem}
\begin{proof}
    Suppose $F$ has more than 2 elements. Then, we may choose an $h \in F \setminus\{0,1\}$ and write $I_n(r)$ as a sum of 2 matrices.
    \begin{align*}
        I_n(r)
        &= \diag(\underbrace{1,\ldots,1}_{r-k},\underbrace{1-h,\ldots,1-h}_{m-(r-k)},\underbrace{0,\ldots,0}_{r-m},\underbrace{0,\ldots,0}_{n-r})\\
        &+ \diag(\underbrace{0,\ldots,0}_{r-k},\underbrace{h,\ldots,h}_{m-(r-k)},\underbrace{1,\ldots,1}_{r-m},\underbrace{0,\ldots,0}_{n-r}).
    \end{align*}

    Let $X_m, X_k$ be the two diagonal matrices in the equation above. Then, $X_m$ has $m$ non-zero rows and thus has rank $m$ while $X_k$ has $k$ non-zero rows and thus has rank $k$.

    If $F$ is the field with two elements, let $U$ be an $m \times m$ upper triangular matrix with exactly $m+k-r$ ones placed along its superdiagonal, and zeros elsewhere.
    \[U = \sum_{i=1}^{m+k-r} E_{i, (i+1)}.\]
    We note that $U$ has $m+k-r$ pivot positions and thus has rank $m+k-r$. We then construct the submatrices $X_m,X_k$ as follows:
    \[
    X_m =
    \begin{bmatrix}
        I_m - U & 0\\
        0 & 0_{r-m}
    \end{bmatrix},
    X_k =
    \begin{bmatrix}
        U & 0\\
        0 & I_{r-m}
    \end{bmatrix}.
    \]
    The block $I_m - U$ is upper triangular with exactly $m$ ones on its diagonal, giving it a rank of exactly $m$. 
    Utilizing the properties of block diagonal matrices,
    \[\rank(X_m) = \rank(I_m - U) = m,\]
    \[\rank(X_k) = \rank(U) + \rank(I_{r-m}) = m+k-r + r -m = k.\]
    A direct computation shows that $X_m,X_k$ add up to $I_r$.
    \[
    X_m + X_k =
    \begin{bmatrix}
        I_m - U & 0\\
        0 & 0_{r-m}
    \end{bmatrix}
    +
    \begin{bmatrix}
        U & 0\\
        0 & I_{r-m}
    \end{bmatrix}
    =
    \begin{bmatrix}
        I_m & 0\\
        0 & I_{r-m}
    \end{bmatrix}
    = I_r.
    \]
    Both $X_m$ and $X_k$ are then padded with zeros to reach the dimension $n\times n$.
\end{proof}
Let us demonstrate Lemma~\ref{lem:r>m} with an example.
\begin{expl}
    In the ring $M_4(\F_2)$, we can write $I_4(4)$ as a sum of matrices with rank $m=3$ and rank $k=2$.
    \[
    X_m =
    \begin{bmatrix}
        \mathbf{1} & 1 & 0 & 0\\
        0 & \mathbf{1} & 0 & 0\\
        0 & 0 & \mathbf{1} & 0\\
        0 & 0 & 0 & 0
    \end{bmatrix},
    X_k =
    \begin{bmatrix}
        0 & 1 & 0 & 0\\
        0 & 0 & 0 & 0\\
        0 & 0 & 0 & 0\\
        0 & 0 & 0 & \mathbf{1}
    \end{bmatrix}.
    \]
\end{expl}

\begin{thm} \label{thm:sum_rank_sets}
    An $n \times n$ matrix $A$ with rank $r$ is a sum of two $n \times n$ matrices with rank $k$ and rank $m$ where $m \geq k$ if and only if $m-k \leq r \leq \min\{m+k, n\}$. Consequently 
    \[ V_k + V_m = \bigsqcup_{i=m-k}^{\min\{m+k, n\}} V_i.\]
\end{thm}
\begin{proof}
    By Lemma~\ref{lem:r<m} and Lemma~\ref{lem:r>m}, we have established that if
    \[m-k \leq r \leq \min\{m+k, n\},\]
    then there exist matrices $X_m \in V_m$, $X_k \in V_k$ such that
    \[I_n(r) = X_m + X_k.\]
    As explained above, we can find $P, Q \in GL_n(F)$ such that $A=PI_n(r)Q$ and hence 
    \[A = PI_n(r)Q = P (X_m + X_k)Q = PX_mQ + PX_kQ.\]
    By the properties of $P$ and $Q$, $PX_mQ \in V_m$ and $PX_kQ \in V_k$.
    Therefore, any $n \times n$ matrix $A$ with rank $r$ is a sum of two matrices with rank $k$ and rank $m$ if $m-k \leq r \leq \min\{m+k, n\}$.

    Conversely, suppose that $A$ is an $n \times n$ matrix with rank $r$ and $X_m,X_k$ are $n \times n$ matrices with rank $m, k$ respectively such that
    \[A = X_m + X_k.\]
    By the rank subadditivity property, 
    \[r:= \rank(A) = \rank(X_m+X_k) \leq \rank(X_m)+\rank(X_k) = m+k .\]
    Since $A \in M_n(F)$, we conclude that $r \leq \min \{m+k, n\}.$ Similarly, let us consider
    \[X_m = (X_m + X_k) + (-X_k)\]
    \[X_m = A + (-X_k).\]
    By the rank subadditivity property, $m \leq r + k$, and therefore $m-k \leq r.$ In summary, we have 
    \[m-k \leq r \leq \min\{m+k, n\}.\]
\end{proof}
Additionally, numerical experiments validating Theorem~\ref{thm:sum_rank_sets} can be found at \cite{githubrepo}.
\begin{cor} \label{cor:multiply}
    Let $1 \leq k \leq n$ and $m \geq 2$. Then, any matrix with rank $r$ with $ r \leq m  k$ can be written as a sum of $m$ matrices with rank $k$. In other words
    \[\forall m \geq 2, \underbrace{V_k + \cdots + V_k}_{m} = \bigsqcup_{i=0}^{\min\{mk, n\}} V_i.\]
\end{cor}
\begin{proof}
    We proceed by induction on the number of summands. The base case $m=2$ is true by Theorem~\ref{thm:sum_rank_sets}. Assume that
    \[ \underbrace{V_k + \ldots + V_k}_{m-1} = \bigsqcup_{i=0}^{\min\{(m-1)k, n\}} V_i.\]
    Considering the sum of $m$ copies of $V_k$,
    \[\underbrace{V_k + \ldots + V_k}_{m} = \underbrace{V_k + \ldots + V_k}_{m-1} + V_k = \left( \bigsqcup_{i=0}^{\min\{(m-1)k, n\}} V_i \right)+ V_k.\]
    By the properties of our set addition operation in Definition~\ref{def:set_addition},
    \[\left( \bigsqcup_{i=0}^{\min\{(m-1)k, n\}} V_i \right)+ V_k = \bigcup_{i=0}^{\min\{(m-1)k, n\}} ( V_i + V_k ).\]
    Applying Theorem~\ref{thm:sum_rank_sets} to each pair, we get
    \[\bigcup_{i=0}^{\min\{(m-1)k, n\}} ( V_i + V_k ) = \bigcup_{i=0}^{\min\{(m-1)k, n\}} \left( \bigcup_{j=|i-k|}^{\min\{i+k, n\}} V_j \right) = \bigcup_{i=0}^{\min\{mk, n\}} V_i.\]
    Therefore,
    \[\forall m \geq 2, \underbrace{V_k + \ldots + V_k}_{m} = \bigsqcup_{i=0}^{\min\{mk, n\}} V_i.\]
\end{proof}

We now describe the distance between any two matrices in the gcd-graph $\Gamma(M_n(F), S).$ We recall that  $S = \bigsqcup_{i=1}^r V_{k_i}$. Let $k_{max}:=k_r$ be the highest rank among matrices in $S$.
\begin{thm} \label{thm:distance}
    The distance between any two distinct and non-adjacent matrices $A,B$, is
    \[d(A,B) = \max \left\{2, \left\lceil \frac{\rank(B-A)}{k_{max}} \right\rceil \right\}.\]
\end{thm}
\begin{proof}
    Let $A,B$ be two non-adjacent matrices. In other words, $B-A \not \in S$. A path of length $d$ between matrices $A$ and $B$ implies that $B - A$ can be written as a sum of $d$ elements, $M_1, M_2, \ldots, M_d \in S$.
    \[\rank(B-A) = \rank \left(\sum_{i=1}^{d} M_i\right).\]
    By the subadditivity of the rank, we have 
    \[ \rank(B-A) \leq \sum_{i=1}^d \rank(M_i) \leq d k_{max}. \] 
     This shows that $d \geq \rho$ where 
    \[\rho = \frac{\rank(B-A)}{k_{max}}.\]
    We divide the problem into two cases. Either $\rho \leq 2$ or $\rho>2$. If $\rho \leq 2$,
    \[\rank(B-A) \leq 2 \cdot k_{max}.\]
    By Theorem~\ref{thm:sum_rank_sets}, $B-A$ can be written as a sum of two matrices with rank $k_{max}$. Thus,
    \[d(A,B) = 2.\]
    If $\rho>2$,
    \[\rank(B-A) \leq \lceil \rho \rceil \cdot k_{max}.\]
    By Corollary~\ref{cor:multiply}, $B-A$ can be written as a sum of $\lceil \rho \rceil$ matrices with rank $k_{max}$. Thus,
    \[d(A,B) = \lceil \rho \rceil.\]
    Therefore, for any two non-adjacent matrices $A,B$,
    \[d(A,B) = \max \left\{2, \left\lceil \frac{\rank(B-A)}{k_{max}} \right\rceil \right\}.\]
\end{proof}

Combining Corollary \ref{cor:multiply} and Theorem \ref{thm:distance}, we have the following.
\begin{thm} \label{cor:diameter}
Let \(G:=\Gamma(M_n(F), [k_1, k_2, \ldots, k_r])\) be a gcd-graph with \(1 \leq r \leq n-1\). Then $G$ is connected. Furthermore, the diameter is attained by a pair of matrices \(A,B\) for which \(\rank(B-A)\) is the largest positive integer \(k \leq n\) such that \(k \notin [k_1, k_2, \ldots, k_r]\). In particular,
\[
\diam(\Gamma) = \max \left\{2, \left\lceil \frac{n}{k_{\max}} \right\rceil \right\}.
\]
\end{thm}

\subsection{Bipartite and Hamiltonian gcd-graphs} \label{subsec:bipartite}
In this section, we study some further graph-theoretic properties of gcd-graphs over $M_n(F)$. First, we show that these graphs are never bipartite (to avoid trivial cases, we restrict ourselves to non-empty and non-complete gcd-graphs).
\begin{thm} \label{thm:bipartite}
    A gcd-graph $\Gamma(M_n(F), S)$ is not bipartite.
\end{thm}
\begin{proof}
    We have established in Corollary ~\ref{cor:generating_set} that $S$ must be a non-empty disjoint union of some $V_k$'s. Let $A \in V_k$ be a matrix of rank $k$. By Corollary~\ref{cor:multiply}, $A$ can be written as a sum of 2 matrices $B,C \in V_k$; namely $A = B + C$ and hence  $C = A-B.$ Since $C \in V_k$, $A$ and $B$ are adjacent. Consider
    \[\rank(A-0_n) = \rank(A) = k.\]
    Thus, $A$ and $0_n$ are adjacent. Similarly, $B$ and $0_n$ are also adjacent. This establishes an odd cycle with 3 vertices $A,B,0_n$. Therefore, $\Gamma(M_n(F), S)$ is not bipartite.
\end{proof}

\begin{thm}
    A gcd-graph $\Gamma(M_n(\F_q), S)$ is Hamiltonian.
\end{thm}
\begin{proof}
By definition, the underlying vertex set of the gcd-graph $\Gamma(M_n(\F_q), S)$ is  defined over the abelian group $(M_n(\F_q),+)$. More precisely, the vertices of $\Gamma(M_n(\F_q),S)$ are connected based on matrix addition and subtraction. Because matrix addition is commutative, the underlying additive group of $M_n(\F_q)$ is a finite abelian group. Furthermore, since $n > 1$ and $q \ge 2$, the number of distinct matrices in the group is
    \[|M_n(\F_q)| = q^{n^2} \geq 16.\]
    By \cite[Corollary 3.2]{MARUSIC198349}, we know that every connected Cayley graph of an abelian group of order at least three is Hamiltonian. By  Theorem \ref{cor:diameter}, $\Gamma(M_n(\F_q), S)$ is connected with more than $3$ vertices.  Therefore, $\Gamma(M_n(\F_q), S)$ is Hamiltonian.
\end{proof}

\subsection{Clique numbers} \label{subsec:rank1}
In this section, we study the clique number of a gcd-graph. In the special cases where the gcd-graph is $\Gamma(M_n(\F_q), [n])$, the clique number is calculated explicitly in \cite{kiani2012unitary}. In this section, we study the clique number of the gcd-graph $\Gamma(M_n(\F_q), [1]).$ We begin with a lemma about rank $1$ matrices.

\begin{lem} \label{lem:space}
    If $S \subseteq V_1$ is a set of rank 1 matrices such that
    \[\forall A \neq B \in S, \rank(A - B) = 1,\]
    then all matrices in $S$ share either the same column space or the same row space.
\end{lem}
\begin{proof}
    Note that any rank 1 $n \times n$ matrix $A$ can be decomposed into
    \[A = uv^T,\]
    with $u$ and $v$ being vectors of dimension $n$. Let $A,B \in S$ be two arbitrary matrices such that $A= x_A y_A^T$ and $B = x_B y_B^T$. Then, $A - B$ has rank 1.
    \[A - B = x_A y_A^T - x_B y_B^T = uv^T.\]
    Assume by way of contradiction that $x_{A},x_{B}$ are linearly independent and $y_{A},y_{B}$ are linearly independent. Then, there exists a vector $w\in F^n$ such that $y_{A}^T w \neq 0$ and $y_{B}^T w = 0$. We then have 
    \[x_{A} \left( y_{A}^T w \right) - x_{B} \left( y_{B}^T w \right) = u \left( v^T w \right)\]
    \[ \implies x_{A} \left( y_{A}^T w \right) = u \left( v^T w \right)\]
    \[ \implies x_{A} = u \left( v^T w \right)\left( y_{A}^T w \right)^{-1}\]
    \[ \implies x_{A} \in \spn(u).\]
    Similarly, by an identical argument, we have     $x_{B} \in \spn(u).$  This is a contradiction because $x_{A},x_{B} \in \spn(u)$ despite $x_{A},x_{B}$ being linearly independent. Thus, $x_{A},x_{B}$ and $y_{A},y_{B}$ cannot both be linearly independent . This argument can be expanded to show that all matrices in $S$ either share the same column space or the same row space. Again, assume by way of contradiction that there exists $A,B,C \in S$ such that $A,B$ share the same column space and $A,C$ share the same row space. Suppose
    \[A=uv^T,B=uy^T,C=xv^T,\]
    where $y,v$ are linearly independent and $x,u$ are linearly independent. Since $B,C \in S$,
    \[\rank(B-C) = \rank(uy^T - xv^T) =1.\]
    This only happens if $y \in \spn(v)$ or $x \in \spn(u)$, leading to a contradiction. Thus, it is impossible to have a triplet of rank 1 matrices where the differences are all rank 1 without all three matrices aligning on the same column space or the same row space. By induction, this property holds for the entire set $S$.

\end{proof}

\begin{thm} \label{thm:clique}
    The clique number of $\Gamma(M_n(F), [1])$ is $|F|^n.$
\end{thm}
\begin{proof}
    Let $\mathcal{M}$ be a maximal clique in $\Gamma(M_n(F), [1])$. 
Let $A_0$ be an arbitrary member matrix of $\mathcal{M}$. Then, by Proposition ~\ref{prop:automorphism}, the translation map $t_{-A_0}$ is an automorphism of $\Gamma(M_n(F),[1]).$ Then, $\mathcal{M'} = t_{-A_0}(\mathcal{M})$ is also a maximal clique in $\Gamma(M_n(F), [1])$. Furthermore, we know that $0_n \in \mathcal{M'}$ and 
    \[\forall A' \in \mathcal{M'} \setminus \{0_n\}, \rank(A')=\rank(A'-0_n)=1.\]
    We have shown that for all non-zero matrices $A',B' \in \mathcal{M'}$, $A',B',$ and $A'-B'$ both have rank 1. Lemma~\ref{lem:space} establishes that this is only possible when all matrices in $\mathcal{M'}$ share the same column space or the same row space. Note that $0_n$ is a special case of the decomposition $A=uv^T$ with either $u$ or $v$ being the null vector. Then, either
    \[\exists u | \forall A' \in \mathcal{M'}, A' = u y_{A'}^T,\]
    or
    \[\exists v | \forall A' \in \mathcal{M'}, A' = x_{A'} v^T.\]
 We conclude that  the size of the maximal clique is the number of distinct $n$-dimensional vectors over the field $F$.
    \[\omega(\Gamma(M_n(F), [1])) = |\mathcal{M'}| = |F|^n.\]
\end{proof}

\subsection{Sums of elements in $V_k$} \label{subsec:sum}
In this section, we study the sum of elements in $V_k$. We have the following result. 
\begin{thm} \label{thm:sum}
    The sum of all matrices in $M_n(F)$ with fixed rank $k$ is the zero matrix.
    \[\sum_{A \in V_k} A = 0_n.\]
\end{thm}
\begin{proof}

We provide two arguments for this theorem. One works when $|F|>2$ and the other works for all cases. First, suppose that $|F|>2$. Then, there exists a $\lambda$ such that $\lambda \in F \setminus \{0,1\}.$
    Since scalar multiplication does not change the rank of a matrix, we have $\lambda A \in V_k$ for all  $A \in V_k$. Therefore
    \[\sum_{A \in V_k} A = \sum_{A \in V_k} (\lambda A) = \lambda \left( \sum_{A \in V_k} A \right).\]
Therefore 
    \[(1- \lambda) \left( \sum_{A \in V_k} A \right) = 0_n.\]
    Since $\lambda \neq 1$, $1-\lambda \neq 0$. It follows that $\sum_{A \in V_k} A = 0_n.$

The second proof works for all $F.$ In fact, let $P \in GL_n(F)$ be an elementary matrix representing an arbitrary row operation. Multiplication with $P$ on the left does not change the rank of the matrix. In other words for each  $A \in V_k$ we also have  $PA \in V_k.$  Let $M$ be the sum of all matrices in $M_n(F)$ with fixed rank $k$. We then have
    \[M =\sum_{A \in V_k} A = \sum_{A \in V_k} (PA) = P \left( \sum_{A \in V_k} A \right) = PM.\]
    In other words, applying any row operation on $M$ does not change $M$. The only matrix with this property is $0_n$. This completes the proof.
\end{proof}

\begin{rem}
Let \( G \) be a gcd-graph over \( M_n(\mathbb{F}_q) \). An interesting corollary of Theorem \ref{thm:sum} is that the walk obtained by successively applying each element of the generating set exactly once, in any order, returns to the starting vertex.
\end{rem}
\section*{Acknowledgements}
We thank Professor Nguyen Duy Tan for his insightful comments and valuable suggestions, which significantly improved both the clarity and overall quality of this work. Parts of this paper were presented at the 2026 Annual Meeting of the Illinois Section of the Mathematical Association of America (MAA), held in Springfield, Illinois. The first-named author gratefully acknowledges the organizers for providing a travel grant that supported the presentation of this work at the meeting. Last but not least, we would like to thank the referee and the editor for their suggestions,  which improved our exposition.


\begin{thebibliography}{10}

\bibitem{unitary}
Reza Akhtar, Megan Boggess, Tiffany Jackson-Henderson, Isidora Jim{\'e}nez, Rachel Karpman, Amanda Kinzel, and Dan Pritikin, \emph{On the unitary {Cayley} graph of a finite ring}, The Electronic Journal of Combinatorics \textbf{16} (2009), no.~1, R117.

\bibitem{anderson2021graphs}
David Anderson, Thangaraj Asir, Ayman Badawi, and Tamizh Chelvam, \emph{Graphs from rings}, Springer, 2021.

\bibitem{chen2022unitary}
Bocong Chen and Jing Huang, \emph{On unitary {Cayley} graphs of matrix rings}, Discrete Mathematics \textbf{345} (2022), no.~1, 112671.

\bibitem{guterman2021}
Alexander Guterman, Marianne Johnson, Mark Kambites, and Artem Maksaev, \emph{Linear functions preserving Green's relations over fields}, Linear Algebra and its Applications \textbf{611} (2021), 310--333.

\bibitem{kiani2012unitary}
Dariush Kiani and Mohsen Molla~Haji Aghaei, \emph{On the unitary {Cayley} graph of a ring}, Electron. J. Comb. (2012).

\bibitem{klotz2007some}
Walter Klotz and Torsten Sander, \emph{Some properties of unitary {Cayley} graphs}, The Electronic Journal of Combinatorics \textbf{14} (2007), no.~1, R45.

\bibitem{MARUSIC198349}
Dragan Marušič, \emph{Hamiltonian circuits in {Cayley} graphs}, Discrete Mathematics \textbf{46} (1983), no.~1, 49--54.

\bibitem{pst-gcd-new-2}
Yotsanan Meemark and Songpon Sriwongsa, \emph{Perfect state transfer in unitary {C}ayley graphs over local rings}, Transactions on Combinatorics \textbf{3} (2014), no.~4, 43--54.

\bibitem{Migler01102006}
Theresa Migler, Kent~E. Morrison, and Mitchell Ogle, \emph{How much does a matrix of rank k weigh?}, Mathematics Magazine \textbf{79} (2006), no.~4, 262--271.

\bibitem{minavc2024gcd}
J{\'a}n Min{\'a}{\v c}, Tung~T. Nguyen and Nguyen~Duy T{\^a}n, \emph{On the gcd-graphs over polynomial rings},
Canadian Journal of Mathematics. First View (2025), 1--28.
DOI: \href{https://doi.org/10.4153/S0008414X25101673}{10.4153/S0008414X25101673}.

\bibitem{githubrepo}
Dung Nguyen and Tung~T. Nguyen, \emph{Gcd-graphs over matrix rings}, \url{https://github.com/nguyend77/gcd-graphs-matrix-rings}, 2026.

\bibitem{nguyengcd2026}
Tung~T. Nguyen and Nguyen~Duy T{\^a}n, \emph{On gcd-graphs over finite commutative rings}, Journal of Algebra and Its Applications. In press (2026). DOI: \href{https://doi.org/10.1142/S0219498827501751}{10.1142/S0219498827501751}.

\bibitem{nguyen2026u}
Tung~T Nguyen and Nguyen~Duy T{\^a}n, \emph{On ${U}$-unitary {Cayley} graphs over finite rings}, arXiv preprint arXiv:2603.21239 (2026).

\bibitem{pst-gcd-new}
Issaraporn Thongsomnuk and Yotsanan Meemark, \emph{Perfect state transfer in unitary {C}ayley graphs and gcd-graphs}, Linear and Multilinear Algebra \textbf{67} (2019), no.~1, 39--50.

\end{thebibliography}
\end{document}